\documentclass[reqno]{amsart}
\usepackage[mathscr]{eucal}
\usepackage{MnSymbol}
\usepackage{stmaryrd}
\usepackage{amsbsy}
\usepackage{amsmath}
\usepackage{enumitem}
\usepackage{amsthm}
\usepackage[pagebackref=true, colorlinks]{hyperref}
\usepackage{enumitem}
\usepackage{tikz}
\usetikzlibrary{matrix,arrows,decorations.pathmorphing}
\usepackage{doi}

\newtheorem{theorem}{Theorem}

\newtheorem{corollary}[theorem]{Corollary}

\theoremstyle{definition}
\newtheorem{example}[theorem]{Example}
\newtheorem{remark}[theorem]{Remark}

\newtheorem{claim}[theorem]{Claim}

\newcommand{\inj}{\mbox{\rm Inj}\hspace{0.02in}}

\newcommand{\buletita}{{\tiny$\bullet$}}
\newcommand{\reg}{\mbox{\rm Reg}\hspace{0.02in}}
\newcommand{\sur}{\mbox{\rm Sur}\hspace{0.02in}}

\begin{document}
\title[Surjection and inversion for locally Lipschitz maps]{Surjection and inversion for locally Lipschitz maps between Banach spaces}

\author{Olivia Gut\'u and Jes\'us A. Jaramillo}

\address{O. Gut\'u: Departamento de Matem\'aticas \\ Universidad de Sonora \\  83000 Hermosillo, Sonora,  M\'exico}

\email{oliviagutu@mat.uson.mx}

\address{J. A. Jaramillo: Instituto de Matem\'atica Interdisciplinar (IMI) \\  Universidad Complutense de Madrid\\ 28040 Madrid, Spain}

\email{jaramil@mat.ucm.es}

\thanks{Research supported in part by MICINN under Project MTM2015-65825-P (Spain).}



%\date{\today}

\maketitle

%%%%%%%%%%%%%%%%%%%%%%%%

\begin{abstract}

We study the global invertibility of non-smooth, locally Lipschitz maps between infinite-dimensional Banach spaces, using a kind of Palais-Smale condition. To this end, we consider the Chang version of the weighted Palais-Smale condition for locally Lipschitz functionals in terms of the Clarke subdifferential, as  well as the notion of pseudo-Jacobians in the infinite-dimensional setting, which are the analog of the pseudo-Jacobian matrices defined by Jeyakumar and Luc. Using these notions,  we derive our results about existence and uniqueness of solution for nonlinear equations. In particular, we give a version of the classical Hadamard integral condition for global invertibility in this context.

\end{abstract}

%%%%%%%%%%%%%%%%%%%%%%%

\section{Introduction}

Surjectivity and invertibility of maps is an important issue in nonlinear analysis. In a smooth setting, if $f:X\to Y$ is a $C^1$ map between Banach spaces, such that its derivative $f'(x)$ is an isomorphism for every $x\in X$, from the classical Inverse Function Theorem we have that $f$ is locally invertible around each point. If, in addition,  $f$ satisfies the so-called {\it Hadamard integral condition:}
$$
\int_0^{\infty} \inf_{\Vert x \Vert \leq t} \Vert f'(x)^{-1} \Vert^{-1} \, {\rm d}t = \infty,
$$
then $f:X\to Y$ is globally invertible, and thus a global diffeomorphism from $X$ onto $Y$ (see e.g. the paper by Plastock \cite{Plastock} for a proof of this result).  This sufficient condition for global invertibility was first considered by Hadamard \cite{Hadamard} for maps between finite-dimensional spaces,  and has been widely used since then. We refer to the recent survey paper \cite{Gu} for an extensive information about this and other conditions for global invertibility of smooth maps between Banach spaces and, more generally, between
Finsler manifolds.

\

In a nonsmooth setting, if $f: \mathbb R^n \to \mathbb R^n$ is a locally Lipschitz map, Pourciau obtained in \cite{Pourciau1982} and \cite{Pourciau1988} suitable versions of the Hadamard integral condition, using the {\it Clarke generalized Jacobian}. These results have been extended to the setting of finite-dimensional Finsler manifolds in \cite{JaMaSa}. For continuous maps  $f: \mathbb R^n \to \mathbb R^n$ which are not assumed to be locally Lipschitz,  Jeyakumar and Luc introduced in \cite{JL0} the concept of {\it approximate Jacobian matrix,} which was later called {\it pseudo-Jacobian matrix} (see \cite{JL}). A global inversion theorem, with a version of the Hadamard integral condition in this context, is given in \cite{JaMaSa1}.

\

If $f:X\to Y$ is a nonsmooth map between infinite-dimensional Banach spaces, the problem of local invertibility of $f$ is more delicate. Assuming that $f$ is a local homeomorphism, F. John obtained in \cite{John} a global inversion theorem with a suitable version of the Hadamard integral condition in terms of the {\it lower scalar  Dini derivative} of $f$. Later on, Ioffe obtained in \cite{Io} a global inversion result for  a continuous map $f$ which is locally one-to-one, using an analog of Hadamard integral condition, defined in terms of the so-called {\it constant surjection} of $f$ at every point. Further results along this line have been obtained in \cite{GutuJaramillo} and \cite{GaGuJa}  in the more general setting of maps between metric spaces. More recently, in \cite{JaLaMa}, the authors consider the notion of {\it pseudo-Jacobian} for a continuous map between Banach spaces, which is an extension of pseudo-Jacobian matrices of Jeyakumar and Luc to this setting, and obtain different global inversions results in this context.

\

From the purely topological point of view, the Banach-Mazur Theorem states that a local homeomorphism $f$ between Banach spaces is a global one if and only if it is a {\it proper} map, that is, the preimage $f^{-1}$ sends compact sets to compact sets \cite{BANACHMAZUR}.  The proof of this result and the classical global inversion theorems cited above are generally addressed through the use of the path-lifting property  or some other similar approach which includes a monodromy argument. Actually,  Rabier makes a masterful  use of the path-lifting arguments in the smooth setting and get the complete geometric picture leading to a generalization of the classical Ehresmann theorem of differential geometry; see Theorem 4.1  in \cite{RABIER1997}. In particular, the Banach-Mazur Theorem and the Hadamard integral condition were non-trivially related by Rabier via a sort of ``uniform-strong'' Palais-Smale condition so called {\it strong submersion} with uniformly splits kernels.

Some conditions intimately related to strong submersions seem to escape from the monodromy technique.  Nevertheless, for functions between metric spaces,  Katriel   \cite{Ka}  proposes a completely different approach based on an abstract mountain-pass theorem  and the Ekeland variational principle, in order to obtain global inversion theorems in a non-smooth setting,  via the study of the critical points of the functional $x\mapsto d(y,f(x))$ for all $y$ in $Y$. Idczak {\it et al.} \cite{ISW} set down the Katriel approach in the Banach space setting proving  that a local diffeomorphism with a Hilbert target space  is a global diffeomorphism if the functional  $F_y(x)=\frac{1}{2}\pmb{|}f(x)-y\pmb{|}^2$ satisfies the Palais-Smale condition for all $y\in Y$. We refer to \cite{GGS} and \cite{GK-15} for further developments in this direction. See also \cite{GORDON}, an earlier reference in this line in the finite-dimensional setting.

The generalization of Idczak {\it et al.} result for functions between two real Banach spaces and the relationship of this condition with the Hadamard integral condition and the strong submersions of Rabier were established in \cite{GUTU2018} for $C^1$ maps. However the connection of the Palais-Smale conditon for $F_y$ with the nonsmooth global inversion results cited above for a locally Lipschitz map $f$ are the pending issues just addressed in this article.

\section{Calculus with pseudo-Jacobians}
Let $(X,|\cdot|)$ and  $(Y,\pmb{|}\cdot\pmb{|})$ be real Banach spaces and $U$ be a nonempty open subset of $X$. As usual, $L(X,Y)$ and $X^*$ will  denote the space of bounded linear operators from $X$ into $Y$ and the topological dual of $X$,  respectively.

\subsection*{Pseudo-Jacobians}  Let $f:U\rightarrow Y$ be a continuous map. Some important examples of a derivative-like objects for continuous maps can be included in a general frame so called pseudo-Jacobians (see \cite{JaLaMa}). The definition of a pseudo-Jacobian of $f$ at a point $x$ involves an approximation of the ``scalarized'' functions $y^*\circ f$, through all directions $y^*\in Y$ by means of upper Dini directional derivatives and a sublinearization of the approximations by a set of operators. Recall, if $\phi:U\rightarrow \mathbb{R}$ is a real-valued function and $x$ is a point in $U$, the {\it upper right-hand Dini derivative} of $\phi$ at $x$ with respect to a vector $v\in X$ is defined as:
$$\phi'_+(x;v)=\limsup_{t\rightarrow 0^+}\frac{\phi(x+tv)-\phi(x)}{t}.$$
A nonempty subset $Jf(x)\subset L(X,Y)$ is said to be a {\it pseudo-Jacobian} of $f$ at $x\in U$ if, for every $y\in Y^*$ and $v\in X$:
\begin{equation}\label{pseudojacobians}
(y^*\circ f)'_+(x;v)\leq\sup\{\langle y^*,T(v)\rangle:T\in Jf(x)\}
\end{equation}
A set-valued mapping $Jf:U\rightarrow 2^{L(X,Y)}$ is called a {\it pseudo-Jacobian mapping}  for $f$ on $U$ if for every $x\in U$ the set $Jf(x)$ is a pseudo-Jacobian of $f$ at $x$.

\

The following examples of pseudo-Jacobians  for a  continuous function $f:U\rightarrow Y$ between Banach spaces are explained with detail in \cite{JaLaMa}; see also   \cite{JL0} for the finite-dimensional theory of pseudo-Jacobians.

\begin{example} If $f$ is {\it G\^ateaux differentiable} at $x$, then the singleton $Jf(x):=\{df(x)\}$ is a pseudo-Jacobian of $f$ at $x$. In particular, this holds if $f$ is {\it Fr\'echet differentiable} or  {\it strictly differentiable}. Recall, a function $f$ is strictly differentiable at $x$ if there is a continuous linear map, denoted by $df(x)$, such that for every $\epsilon>0$ there is $\varrho>0$ such that if $u,w\in B(x;\varrho)$ then:
$$\pmb{|}f(u)-f(w)-df(x)(u-w)\pmb{|}\leq\epsilon|u-w|.$$
Now suppose that $X=\mathbb{R}^n$ and $Y=\mathbb{R}^m$. If $f$ admits a singleton pseudo-Jacobian at $x$ then  $f$ is G\^ateaux differentiable at $x$ and its derivative coincides with the pseudo-Jacobian matrix. In the infinite-dimensional setting, a continuous map between Banach spaces admits a singleton pseudo-Jacobian at a point $x$ if, and only if, it is  {\it weakly G\^ateaux differentiable} at $x$, see p. 23 in \cite{AUBIN2006}.
\end{example}

\begin{example}\label{balls}
Suppose $f$ is a {\it locally Lipschitz map}, namely,  for every $x\in U$ there exist $L,r>0$ such that, whenever $u,w\in B(x; r)\subset U$:
$$\pmb{|}f(u)-f(w)\pmb{|}\leq L|u-w|$$
Consider the {\it local Lipschitz constant} of $f$ at $x\in U$, given by:
$${\rm Lip} \, f(x)=\inf_{r>0}\sup\left\{\frac{\pmb{|}f(u)-f(w)\pmb{|}}{|u-w|}:u,w\in B(x,r) \mbox{ and } u\neq w\right\}.$$
Then the set $Jf(x):= {\rm Lip} \, f(x) \cdot \overline{B}_ {L(X, Y)} $,  defined as the unit ball centered at zero of radius ${\rm Lip} \, f(x)$ in the space $L(X,Y)$, is a pseudo-Jacobian of $f$ at $x$.
\end{example}

\begin{example} Let $f: U\subset X \rightarrow Y$ be a locally Lipschitz map. Suppose that $X=\mathbb{R}^n$ and $Y=\mathbb{R}^m$.  The {\it Clarke generalized Jacobian}  of  $f$ at $x$ is a pseudo-Jacobian of $f$ at $x$. Recall that the Clarke generalized Jacobian is equivalent to the generalized Jacobian proposed by  Pourciau in \cite{POURCIAU1977}. An extension of Clarke generalized Jacobian, enjoying all the fundamental properties desired from a derivative set,  was proposed by P\'ales and Zeidan in \cite{PZ} to the case when $X$  and $Y$ are infinite-dimensional Banach spaces, and $Y$ is a dual space satisfying the Radon-Nikodym property, in particular when $Y$ is reflexive. Let us recall the definition in this case. Given a finite-dimensional linear subspace $L\subset X$, we say that $f$ is {\it $L$-G\^ateaux-differentiable} at a point $z\in U$ if there exists a continuous linear map $D_L(z):L\to Y$ such that
$$
\lim_{t \to 0} \frac{f(z+tv) - f(z)}{t}= D_L f(z)(v), \quad \text{for every} \quad v\in L.
$$
Denote by $\Omega_L(f)$ the set made up of all points $z\in U$ such that $f$ is $L$-G\^ateaux-differentiable at $z$, and let $\partial_L f(x)$ be the subset of $\mathcal{L}(L,\, Y)$ given by the the formula
$$
\partial_L f(x) := \bigcap_{\delta>0} \overline{co}^{\tiny WOT} \{ D_L f(z) \, : \, z\in B(x, \, \delta) \cap \Omega_L(f) \},
$$
where $\overline{co}^{\tiny WOT}$ denotes the closed convex hull for the weak operator topology on $L(X, Y)$. The \emph{P\'ales-Zeidan generalized Jacobian of $f$ at the point $x$} is then defined as the set:
$$\partial f(x) = \left\{T\in \mathcal{L}(X,\, Y):\,\, T|_L \in \partial_L f(x),\,\, \text{for each finite dimensional subspace}\, L\subset X\right\}.$$
In particular, if $Y$ is reflexive, the  P\'ales-Zeidan  generalized Jacobian is indeed a pseudo-Jacobian.

\end{example}
\begin{example} If $\phi:U\subset X\rightarrow \mathbb{R}$ is a locally Lipschitz map  the {\it Clarke generalized directional derivative of} $f$ at $x$ with direction $v$ is defined by:
$$\phi^\circ(x,v):=\limsup_{z\rightarrow x, t\rightarrow 0^+}\frac{\phi(z+tv)-\phi(z)}{t}$$
It is well known that the map $v\mapsto \phi^\circ(x;v)$ is convex and continuous.
The  {\it Clarke subdifferential of $\phi$ at $x$} is the non-empty $w^*$-compact convex subset of $X^*$ defined as:
$$\partial \phi (x)=\{x^*\in X^*: x^*(v)\leq \phi^\circ(x;v), \hspace{0.2in} \mbox{for all }x\in X\}$$
Suppose that  $Y=\mathbb{R}$. Then the Clarke subdifferential of $f$ at $x$ is a pseudo-Jacobian of $f$ at $x$.
\end{example}

The theory of pseudo-Jacobians includes a sort of mean value theorem, an optimality condition for real-valued functions and some partial results concerning the chain rule (see \cite{JaLaMa}). In order to get desirable local and global surjection and inversion theorems it is necessary the validity of the chain rule for the composition with distance functions.

\subsection*{Chain Rule and Strong Chain Rule conditions}
Let  $(X,|\cdot |)$ and  $(Y, \pmb{|}\cdot\pmb{|})$ be real Banach spaces, $U$ be an open subset of $X$ and $f:U\rightarrow Y$ be a continuous map  with a pseudo-Jacobian map  $Jf$. For every $y\in Y$, consider the functional:
\begin{equation}
\label{funcionFY}F_y(x):=\pmb{|}f(x)-y\pmb{|}
\end{equation}
For every $x\in U$ and $y\neq f(x)$, we shall define the subset of $X^*$:
$$
\Delta F_y(x):=\partial \pmb{|}\cdot \pmb{|}(f(x)-y)\circ\overline{\rm co}(Jf(x)).
$$
According to \cite{JaLaMa}, we say that that $Jf$ satisfies the {\it chain rule condition} on $U$ if, for each $x\in U$ and $y\neq f(x)$, the set $\Delta F_y(x)$ is a $w^*$-closed and convex subset of $X^*$ and is a pseudo-Jacobian of the functional $F_y$.  We shall say that $Jf$ satisfies the {\it strong chain rule condition} if, in addition to the above requirements, we also have that $f$ is locally Lipschitz and $\Delta F_y(x)$ contains the Clarke subdifferential of $F_y$ at $x$.

\begin{example}\label{diff-chain}
If $f:U\subset X\rightarrow Y$ is continuous and G\^ateaux differentiable on all of $U$ then, for every $x\in U$, the pseudo-Jacobian map $Jf(x)=\{df(x)\}$ satisfies the chain rule condition; see Proposition 2.17 of \cite{JaLaMa}. Furthermore, by Theorem 2.3.10 (Chain Rule II) of \cite{Clarkebook}, we have that if  $f$ strictly differentiable, in particular $C^1$, then  $Jf$ satisfies the strong chain rule condition.
\end{example}

\begin{example}\label{paleszeidan}
Let  $f:U\subset X\rightarrow Y$ be a locally Lipschitz map, where $X$ and $Y$  are reflexive Banach spaces and $Y$ is endowed with a $C^1$-smooth norm. Consider $Jf(x):=\partial f(x)$  the P\'ales-Zeidan  generalized Jacobian of $f$ at $x$. From Corollary 2.18 of \cite{JaLaMa}  we have that $Jf$ satisfies the chain rule condition. Furthermore, taking into account that $\partial f(x)$ is a closed convex subset of $L(X, Y)$ and using Theorem 5.2 in \cite{PZ} we deduce that $Jf$ satisfies in fact the strong chain rule condition. Indeed, given $y\in Y$ we denote $g(z):= \pmb{|} z-y \pmb{|}$. Since $g$ is $C^1$ on the set  $\{ z\in Y \, : \, z\neq y\}$ we have that, for $f(x) \neq y$:
$$
\partial F_y (x) = \partial (g\circ f) (x) = dg(f(x)) \circ \partial f (x) = \partial \pmb{|}\cdot \pmb{|}(f(x)-y)\circ \overline{\rm co}(\partial f(x)) = \Delta F_y(x).
$$
\end{example}

\begin{example}\label{strongballs}
Let  $f:U\subset X\rightarrow Y$ be a locally Lipschitz map, where $X$ and $Y$  are reflexive Banach spaces and $Y$ is endowed with a $C^1$-smooth norm, and consider the pseudo-Jacobian $Jf(x):={\rm Lip} \, f(x) \cdot \overline{B}_ {L(X, Y)} $ considered in Example \ref{balls}. Again from Corollary 2.18 of \cite{JaLaMa}  we have that $Jf$ satisfies the chain rule condition. Furthermore,  $Jf$ also satisfies the strong chain rule condition. Indeed, given $y\in Y$, if we denote $g(z):= \pmb{|} z-y \pmb{|}$ as before, taking into account that $\partial f (x) \subset {\rm Lip} \, f(x) \cdot \overline{B}_ {L(X, Y)}$ (see Theorem 3.8 in \cite{PZ}) we have for $f(x) \neq y$:
$$
\partial F_y (x) = \partial (g\circ f) (x) = dg(f(x)) \circ \partial f (x) \subset \partial \pmb{|}\cdot \pmb{|}(f(x)-y)\circ Jf(x) = \Delta F_y(x).
$$
\end{example}

\section{Pseudo-Jacobians and local inverse theorems}
Let $(X,|\cdot|)$ and $(Y,\pmb{|}\cdot\pmb{|})$ be real Banach spaces, $U$ be an open  subset of $X$, and let $f:U\rightarrow Y$ be a locally Lipschitz map with a pseudo-Jacobian $Jf$ satisfying the chain rule condition. If $T:X\rightarrow Y$ is a bounded linear operator, we consider its {\it Banach constant}, see p. 4 in \cite{IOFFE2017}:
 $$C(T)=\inf_{\pmb{|}v^*\pmb{|}_{Y^*}=1}|T^*v^*|_{X^*}.$$
The Banach constant coincides with the quantity $\sigma(T)$ in \cite{PALES1997}  and also with the number $\tau_T$ in \cite{AZE2006}, see also \cite{RABIER1997}.  Recall, the Banach constant $C(T)$ is positive if and only if $T$ is onto; in such case by the Open Mapping Theorem, $T$ is an open  map. There are a large number of nonlinear versions of this openness criterion e.g.  for a strictly differentiable map $f$ (see \cite{Do}): if $C(df(x_0))>0$ then $f$ {\it is open with linear rate around $x_0$}, namely, there exist a neighborhood $V$ of $x_0$ and a constant $\alpha>0$ such that for every $x\in V$ and $r>0$ with $B(x;r)\subset V$:
\begin{equation}\label{metricregularityA}B(f(x);\alpha r)\subset f(B(x; r)).\end{equation}
A natural quantity to consider in the pseudo-Jacobian frame is the following:
$$\sur Jf (x)=\sup_{r>0}\inf\{C(T): T\in{\rm co}\hspace{0.01in} Jf(B(x;r))\}.$$
Of course, if $f$ is strictly differentiable and $Jf(x):=\{df(x)\}$ then we have that  $\sur Jf (x) = C (df(x)).$

From the very definition, it is clear that the functional $\sur Jf: U\rightarrow [0,\infty)$ is lower semicontinuous. On the other hand, if the set valued map $Jf:U\rightarrow 2^{L(X,Y)}$ is upper semicontinuous at a point $x$, from Proposition 3.4 in \cite{JaLaMa} we have that
$$\sur Jf(x)=\inf\{C(T):T\in{\rm co}\hspace{0.01in} Jf(x)\}.$$

\begin{example}\label{sum}
Let $X$ and $Y$  be reflexive Banach spaces, where $Y$ is endowed with a $C^1$-smooth norm. Consider a map $f:U\subset X\rightarrow Y$ of the form $f=f_1+f_2$, where $f_1:U\rightarrow Y$ is $C^1$-smooth and $f_2:U\rightarrow Y$ is locally Lipschitz. From Example \ref{diff-chain} and Example \ref{strongballs},  we see that $Jf(x):= df_1(x) + {\rm Lip} \, f_2(x) \cdot \overline{B}_ {L(X, Y)} $ is a pseudo-Jacobian of $f$ on $U$, satisfying the chain rule condition. Furthermore, it can be checked as before that in fact $Jf$ satisfies the strong chain condition. Indeed, if we denote $\eta(z):= \pmb{|} z \pmb{|}$, using  Theorem 5.2, Corollary 5.4 and Theorem 3.8 in \cite{PZ}) we have that, for $f(x) \neq y$:
$$
\partial F_y (x) = d\eta (f(x) -y) \circ \partial f(x) = d\eta (f(x) -y) \circ (\partial f_1(x) + \partial f_2(x) )
$$
$$
=  d\eta (f(x) -y) \circ ( df_1(x) + \partial f_2(x)) \subset \partial \pmb{|}\cdot \pmb{|}(f(x)-y)\circ Jf(x) = \Delta F_y(x).
$$
On the other hand, it is not difficult to check (see Example 2.6 in \cite{JaLaMa}) that the set-valued map $Jf: U \rightarrow 2^{L(X, Y)}$  is upper semicontinuous on $U$, so from Proposition 3.4 in \cite{JaLaMa} we obtain that for each $x$ in $U$:
$$\sur Jf (x) = \inf\{C(T) \, : \, T\in  Jf(x)\}= \inf \{C(df_1(x) + R) \, : \, \Vert R \Vert \leq {\rm Lip} \, f_2(x)\}.$$
\end{example}

Now, from Theorem 3.1 of \cite{JaLaMa} we have:

\begin{theorem}[{\bf local openess}]\label{localopeness}
Let $(X,|\cdot|)$ and $(Y,\pmb{|}\cdot\pmb{|})$ be Banach spaces, $U$ be an open subset of $X$ and let $f:U\rightarrow Y$ be a continuous map  with  pseudo-Jacobian  $Jf$ satisfying the  chain rule condition on U.  If $\sur Jf(x_0)>0$ then $f$ is open with linear rate around $x_0$. More precisely, for each $0<\alpha<\sur Jf(x_0)$ there is a neighborhood $V$ of $x_0$ such that for every $x\in V$ and $r>0$ with $B(x;r)\subset V$ the inclusion \eqref{metricregularityA} holds.
\end{theorem}

As a counterpart to the Banach constant, we  consider the {\it dual Banach constant} of a bounded linear operator $T:X\rightarrow Y$  (see \cite{IOFFE2017}, p. 5) defined by:
 $$C^*(T)=\inf_{{|u|}_{X}=1}\pmb{|}Tu\pmb{|}_{Y}.$$
Note that $C^*(T)$ coincides with $\sslash T \sslash$,  the co-norm of $T$ considered by Pourciau \cite{Pourciau1982}, \cite{Pourciau1988} in a finite-dimensional setting, and also in \cite{JaLaMa}. If $C^*(T)>0$ then $T$ is one-to-one. Furthermore,  if $T$ is an isomorphism it is not difficult to check that:
$$
C(T)=C^*(T)=\|T^{-1}\|^{-1}.
$$
The natural quantity to consider in the pseudo-Jacobian frame is the following:
$$\inj Jf(x)=\sup_{r>0}\inf\{C^*(T): T\in{\rm co}\hspace{0.01in} Jf(B(x;r))\}.$$

Using the Mean Value Property given in Theorem 2.7 of \cite{JaLaMa} and proceeding as in the proof of Lemma 3.8 of \cite{JaLaMa},  we have the following:

\begin{theorem}[{\bf local injectivity}]
Let $(X,|\cdot|)$ and $(Y,\pmb{|}\cdot\pmb{|})$ be Banach spaces, $U$ be an open subset of $X$, and let $f:U \rightarrow Y$ be a continuous map  with  pseudo-Jacobian $Jf$ satisfying the  chain rule condition on $U$. If $\inj Jf(x_0)>0$ then $f$ is locally one-to-one at $x_0$. More precisely,  for each $0<\alpha<\inj Jf(x_0)$ there exists a neighborhood $V$ of $x_0$ such that for every $u,w\in V$ we have: $$\pmb{|}f(u)-f(w)\pmb{|}\geq \alpha |u-w|.$$
\end{theorem}

Combining these previous results we obtain the local inversion result below (see Theorem 3.5 of \cite{JaLaMa}). It is useful to introduce first the notion of {\it regularity.}

\

\noindent{\bf Regular pseudo-Jacobian and regularity index}. Let  $f:U\subset X\rightarrow Y$ be a continuous map between Banach spaces. We shall say that the pseudo-Jacobian $Jf$ is {\it regular} at a point $x_0\in U$ if, for some $r>0$, every operator $T\in  {\rm co}\hspace{0.01in} Jf(B(x_0;r))$ is an isomorphism and $\reg Jf(x_0)>0$, where $\reg Jf(x_0)$ is the {\it regularity index} of $f$ at $x_0$ defined as:
$$\reg Jf(x_0):=\sur Jf(x_0)=\inj Jf(x_0).$$

\

\begin{theorem}[{\bf Inverse Mapping Theorem}]\label{localinverse}
Let $(X,|\cdot|)$ and $(Y,\pmb{|}\cdot\pmb{|})$ be Banach spaces, $U$ be an open subset of $X$, $x_0\in U$, and let  $f:U\rightarrow Y$ be a locally Lipschitz map  with a pseudo-Jacobian
$Jf$ satisfying the  chain rule condition on $U$.  Suppose  $Jf$ is regular at $x_0$.  Then  $f$ is a bi-Lipschitz homeomorphism around $x_0$. More precisely,  for each $0<\alpha<\reg Jf (x_0)$ there is an neighborhood $V$ of $x_0$ contained in $U$ with the following properties:
\begin{enumerate}
\item[\buletita] The set  $W:=f(V)$ is open in $Y$.
\item[\buletita] The map $f|_{V}:V\rightarrow W$ is a bi-Lipschitz homeomorphism.
\item[\buletita] The map $f|_V^{-1}$ is $\alpha^{-1}$-Lipschitz on $V$.
\end{enumerate}
\end{theorem}

\section{Metric regularity, Lipschitz rate of $f^{-1}$, Ioffe constant of surjection and some global inversion conditions}

Let $(X,|\cdot|)$ and $(Y,\pmb{|}\cdot\pmb{|})$ be real Banach spaces, $U$ be an open subset of $X$, $x_0\in U$, and let  $f:U\rightarrow Y$ be a locally Lipschitz map  with a pseudo-Jacobian
$Jf$ satisfying the  chain rule condition on $U$ .  Suppose  $Jf$ is regular at $x_0$.  By Inverse Mapping Theorem above, $f$ is a bi-Lipschitz homeomorphism around $x_0$ and  by Theorem \ref{localopeness} (local openness)  $f$ is open with linear rate at $x_0$. The supremum of the nonnegative real numbers $\alpha$ such that for some neighborhood $V_{x_0}$, $B(f(x);\alpha r)\subset f(B(x; r))$, for all $x\in V_{x_0}$ and all $r>0$ with $B(x;r)\subset V_{x_0}$ is called the {\it rate of surjection of $f$} near $x_0$ \cite{IOFFE2016} or {\it exact covering bound of $f$ around $x_0$} \cite{MORDUKHOVICH2006}--- denoted by cov$f(x_0)$.
So, in this context, $\mbox{cov} f(x_0)\geq\reg Jf (x_0)$.  On the other hand,
 in \cite{IOFFE1987} Ioffe introduced the modulus of surjection of $f$ at $x_0$, defined for every $r>0$  by
$$
{\rm S }(f,x_0)(r)=\sup\{\rho\geq 0:B(f(x_0);\rho)\subset f(B(x_0;r))\}.
$$
Ioffe considers also the quantity now called {\it Ioffe constant of surjection} of $f$ at $x_0$:
$$
\mbox{sur}(f,x_0)=\sup_{\epsilon>0}\left(\inf\left\{\frac{{\rm S}(f,x_0)(r)}{r}:0<r<\epsilon\right\}\right).
$$
Let $V_{x_0}$  be a neighborhood of $x_0$ and a constant $\alpha>0$ such that for every $x\in V_{x_0}$ and $r>0$ with $B(x;r)\subset V_{x_0}$ we have that $B(f(x); \alpha r)\subset f(B(x;r))$. In particular, there is an $\epsilon>0$, depending on $\alpha$, such that for every $0<r<\epsilon$: $$B(f(x_0); \alpha r)\subset f(B(x_0;r)).$$ Therefore, for all $0<r<\epsilon$ we have that $\alpha r \leq {\rm S}(f,x_0)(r)$. So,
$$\alpha\leq\inf\left\{\frac{{\rm S}(f,x_0)(r)}{r}:0<r<\epsilon\right\}\leq \mbox {sur}(f,x_0).$$
Finally we have that if $\reg Jf(x_0)>0$ then:
\begin{equation}\label{desigualdad2}
\mbox {sur}(f,x_0)\geq \mbox{cov} f(x_0)\geq\reg J f(x_0).
\end{equation}
If $f$ is $C^1$  and $df(x_0)$ is a linear isomorphism, actually we have:
$$
\mbox {sur}(f,x_0)= \mbox{cov} f(x_0)=\reg Jf(x_0)=C(df(x_0))=\|df(x_0)^{-1}\|^{-1}.
$$
There exists a close relationship between the rate of surjection of $f$ and the so called {\it Lipschitz rate} of the multivalued function $f^{-1}: Y\rightarrow X$ given by
$$
f^{-1}(y)=\{x\in X: y=f(x)\}.
$$
Recall, $f^{-1}$ has the {\it Aubin property} (or {\it pseudo-Lipschitz property}, or {\it Lipschitz-like property}) around $(y_0,x_0)$ if there exists neighborhoods $W_{y_0}$ and $V_{x_0}$ and a number $\mu>0$ such that $\mbox{dist}(x,f^{-1}(y))\leq\mu \, \pmb{|}y-z\pmb{|}$ provided $y,z\in W_{y_0}, x\in f^{-1}(z)\cap V_{x_0}$. The infimum of such $\mu$
 is the {\it Lipschitz rate of $f^{-1}$ near $(y_0,x_0)$} ---denoted by lip$f^{-1}(y_0|x_0)$.   In general we have that $f$ is open with linear rate around $x_0$ if and only if $f^{-1}$ has the Aubin property near $(f(x_0),x_0)$ if and only if $\mbox{cov} f(x_0)>0$. In this case, see Proposition 2.2  in \cite{IOFFE2016}:
\begin{equation}\label{metricregularity}
\mbox{lip}f^{-1}(f(x_0)|x_0)=\mbox{cov} f(x_0)^{-1}
\end{equation}
Now,  let $m$ be a positive lower semicontinuous function on $[0,\infty)$ such that

\

\begin{enumerate}[label=$(A)$]
\item\label{Ioffecondition} For all $x\in X$,  $\sur Jf(x)\geq m(|x|)$.
\end{enumerate}

\

The essential example that comes from Hadamard's original global inversion theorem of 1906 \cite{Hadamard} is associated to the  nonincreasing function $\mu$ on $[0, \infty)$ given by
\begin{equation}\label{Hadamardquantity}\mu(\rho)=\inf_{|x|\leq\rho} \sur Jf(x).\end{equation}
Indeed,  $\mu$ has countably many (jump) discontinuities; we then set $m(\rho):=\mu(\rho)$
if $\mu$ is continuous at $\rho$ and set $m(\rho):=\lim_{t\rightarrow\rho^+}\mu(t)$ if $\mu$ has a jump discontinuity at $\rho$. So, the mapping $m$ is lower semicontinuous on $[0,\infty)$ and satisfies \ref{Ioffecondition} since $\mu(\rho)\geq m(\rho)$ for all $\rho>0$. If in addition, $\mu(\rho)>0$ for all $\rho>0$ then $m$ is positive on $[0,\infty)$.

This argument  shows that, if $\mu$ is a positive nonincreasing function such that, for all $x\in X$,  $\sur Jf(x)\geq \mu(|x|)$, then there exists an associated positive lower semicontinuous function $m$,  constructed as above, such that \ref{Ioffecondition} holds.

In general, if $m$ is a positive lower semicontinuous function satisfying condition \ref{Ioffecondition}, then for all $r>0$ we have $\alpha_r:=\inf \{ m(\rho) \, : \, 0\leq \rho\leq r \}>0$. Therefore for all $x\in B(0;r)$, $\sur J f(x)\geq \alpha$. Thus $f$ is open with linear rate around every $x\in B(0;r)$ with uniform lower bound of the rate of surjection of $f$ on $B(0; r)$. Since $r$ is arbitrary,  condition \ref{Ioffecondition} is actually a global condition.

\

Furthermore, by \eqref{desigualdad2} and Theorem 1 of  \cite{Io} we have the following:

\begin{theorem}[{\bf Global Surjection Theorem}]
Let $(X,|\cdot|)$ and $(Y,\pmb{|}\cdot\pmb{|})$ be Banach spaces and let  $f:X\rightarrow Y$ be a locally Lipschitz map  with a pseudo-Jacobian
$Jf$ satisfying the  chain rule condition on $X$.  Suppose  that condition \ref{Ioffecondition} holds for some  positive lower semicontinuous function $m$ on $[0, \infty)$. Then $f$ is open with linear rate at every $x\in X$, and for each $r>0$ we have:
\begin{equation}
\label{Ioffesurjectionproperty}B(f(0);{\varrho(r)})\subset f(B(0;r)),
\end{equation}
where
$$\varrho(r)=\int_0^r m(\rho)d\rho.$$
Furthermore, $f:X\rightarrow Y$ is surjective provided that, in addition,
$$
\int_0^\infty m(\rho) d\rho =\infty.
$$
\end{theorem}

\

Note that if $\mu$ is the nonincreasing  function given by \eqref{Hadamardquantity}, and $m$ is the associated lower semicontinuous function defined as above, then $\varrho(r)=\int_0^r m(\rho)d\rho=\int_0^r \mu(\rho)d\rho$. Besides, $\mu(\rho)>0$ for all $\rho>0$ if

\

\begin{enumerate}[label=$(B)$]
\item\label{Hadamardcondition} $\int_0^\infty \inf_{|x|\leq\rho} \sur Jf(x) d\rho=\infty.$
\end{enumerate}

\

Therefore we conclude:

\begin{corollary}
Let $(X,|\cdot|)$ and $(Y,\pmb{|}\cdot\pmb{|})$ be Banach spaces and let  $f:X\rightarrow Y$ be a locally Lipschitz map  with a pseudo-Jacobian
$Jf$ satisfying the  chain rule condition on $X$.  Suppose that condition \ref{Hadamardcondition} is satisfied.
Then $f$ is a surjective map, open with linear rate at every $x\in X$, such that for every $r>0$ inclusion \eqref{Ioffesurjectionproperty} holds with
$\varrho(r)=\int_0^r \inf_{|x|\leq\rho} \sur Jf(x) d\rho.$
\end{corollary}

Now, suppose that $f$ is  also locally one-to-one at every $x\in X$, e.g. under hypothesis of Inverse Mapping Theorem above, then by Theorem 2 of \cite{Io} $f$ is actually a global homeomorphism onto $Y$. Note that if $f$ is a global homeomorphism, then  for all $y\in Y$:
\begin{equation}\label{lipschitzinequality}
{\rm lip} \,f^{-1}(y|f^{-1}(y))\geq{\rm Lip} \, f^{-1}(y).
\end{equation}
So we get the following extension of Theorem 3.9 of \cite{JaLaMa}. Note that  if $f$ has a global inverse, property by \eqref{Ioffesurjectionproperty} implies that $f$ is a {\it norm-coercive map}, namely $\lim_{|x|\rightarrow\infty}\pmb{|} f(x)\pmb{|}=\infty$.

\begin{theorem}[{\bf Global Inverse Theorem I}]\label{global-1}
Let $(X,|\cdot|)$ and $(Y,\pmb{|}\cdot\pmb{|})$ be Banach spaces and let  $f:X\rightarrow Y$ be a locally Lipschitz map  with a pseudo-Jacobian $Jf$ satisfying the  chain rule condition on $X$.  Suppose that  $Jf$ is regular at every $x\in X$ and:
\begin{enumerate}[label=$(A')$]
\item\label{Ioffeconditionprima} For all $x\in X$,  $\reg Jf(x)\geq m(|x|)$.
\end{enumerate}
 for some  positive lower semicontinuous function $m$ such that $\int_0^\infty m(\rho) d\rho =\infty$. Then $f$ is a norm-coercive  global homeomorphism onto $Y$ and the inverse $f^{-1}$ is Lipschitz  on bounded subsets of $Y$, and  such that for every $y=f(x)\in Y$: $${\rm Lip} \, f^{-1}(y)\leq(m(|f^{-1}(y)|))^{-1}.$$
 \end{theorem}

\begin{proof}
It only remains to show that the global inverse map $f^{-1}$ is Lipschitz  on bounded subsets of $Y$. Let $R>0$ be given, and consider $r>0$ such that
$$
\int_0^r m(\rho) d\rho > R.
$$
From \ref{Ioffesurjectionproperty}, we have that $B(f(0); R)\subset f(B(0;r))$. As we have remarked before, $\alpha_r:=\inf \{ m(\rho) \, : \, 0\leq \rho\leq r \}>0$, and thus $\reg Jf(x)\geq m(|x|)\geq \alpha_r>0$ whenever $|x| \leq r$.  Therefore, if we fix $0< \alpha < \alpha_r$, we obtain from Theorem \ref{localinverse} that  $f^{-1}$ is locally $\alpha^{-1}$-Lipschitz on the open ball $B(f(0); R)$. Using the convexity of the ball, a standard argument gives that $f^{-1}$ is in fact $\alpha^{-1}$-Lipschitz on the  ball $B(f(0); R)$, and this concludes the proof.
\end{proof}

If $f:X\rightarrow Y$ is a locally Lipschitz map between reflexive Banach spaces such that the P\'ales-Zeidan generalized Jacobian  $\partial f$ is upper semicontinuous,  then the hypotheses of Global Inverse Theorem I are satisfied if for some positive lower semicontinuous function $m$ and each $x\in X$,  every $T\in \partial f(x)$ is an isomorphism and  satisfies $C^*(T)\geq m(|x|)$.  In particular Corollary 3.10 of \cite{JaLaMa} can be deduced from above result. For a $C^1$ map $f:X\rightarrow Y$,  the hypotheses of Global Inverse Theorem I are satisfied if  for some positive lower semicontinuous function $m$ and for each $x\in X$, we have that $df(x)$ is a linear isomorphism and ${\rm cov } f(x)=C^*(df(x))\geq  m(|x|)$.

\section{Palais-Smale condition and locally bi-Lipschitz homeomorphisms}

Let $(X,|\cdot|)$ be a real Banach space and let $F:X\rightarrow \mathbb{R}$ be a locally Lipschitz functional. We define the lower semicontinuous function:
$$\lambda_F(x)=\min_{w^*\in\partial F(x)}|w^*|_{X^*}$$
By a {\it weight} we mean a continuous nondecreasing function $h:[0,+\infty)\rightarrow[0,+\infty)$ such that $$\int_0^\infty \frac{1}{1+h(\rho)}d\rho=+\infty.$$

\

\noindent{\bf Weighted Chang-Palais-Smale condition.} Following Chang \cite{Chang}, we say that the functional $F:X \rightarrow \mathbb R$ satisfies the weighted Chang-Palais-Smale condition with respect to a weight $h$  if any sequence $\{x_n\}$ in $X$ such that $\{F(x_n)\}$ is bounded and \begin{equation}\label{limiteacero}\lim_{n\rightarrow \infty}\lambda_F(x_n)(1+h(|x_n|))=0\end{equation} contains a (strongly) convergent subsequence.

\

Naturally, the limit of a  converging weighted Chang-Palais-Smale sequence must be a critial point, in the sense that $\lambda_F(x)=0$. Furthermore,  if $F$ is bounded from below then, by the Ekeland Variational Principle there exists always a minimizing weighted Chang-Palais-Smale sequence \cite{GUTU2018}. In other words, for any weight $h$:
\begin{itemize}
\item[\buletita] If  $\{x_n\}\subset X$ is a sequence such that $\lim_{n\rightarrow\infty}x_n=\hat x$ and satisfying \eqref{limiteacero} for $h$, we have that  $\lambda_F(\hat x)=0$.
\item[\buletita] If  $F$ is  bounded from below then for $h$ then there is a  sequence $\{x_n\}$ such that $\lim_{n\rightarrow\infty}F(x_n)=\inf_X F$ and satisfying \eqref{limiteacero}.
\end{itemize}

\
\noindent Let  $f:X\rightarrow Y$ be a  locally Lipschitz function and $y\in Y$ be fixed. Consider the functional $F_y(x):=\pmb{|}f(x)-y\pmb{|}$ defined in \eqref{funcionFY}. Suppose that:

\

\begin{enumerate}[label=$(C)$]
\item\label{condicionchang} The locally Lipschitz functional $F_y$ satisfies the weighted Chang-Palais-Smale condition for some  weight $h$.
\end{enumerate}

\

Then the first property above gives us the existence of a minimizing sequence converging to a critical point of $F_y$, so there exists a solution of the non-linear equation $f(x)=y$. The global injectivity comes from the second property above and a  mountain-pass theorem if $f$ has appropriate local properties e.g. under hypothesis of Inverse Mapping Theorem above.  So we have the following extension of Theorem 1 of \cite{GUTU2018} given for  $C^1$ mappings,  in turn, a generalization of Theorem 3.1 of \cite{ISW} for $Y$  Hilbert space and $h = 0$; see Remark \ref{aclaracion} below.

%%%MAIN THEOREM%%%
\begin{theorem}\label{main}
Let $(X,|\cdot|)$ and $(Y,\pmb{|}\cdot\pmb{|})$ be Banach spaces and let  $f:X\rightarrow Y$ be a locally Lipschitz map  with a pseudo-Jacobian $Jf$ regular at every $x\in X$ and satisfying the strong chain rule condition on $X$. Suppose that for some $y\in Y$, the functional $F_y(x)=\pmb{|}f(x)-y\pmb{|}$ satisfies  \ref{condicionchang}. Then there exists a unique solution of the nonlinear equation $f(x)=y$.
\end{theorem}

%%%%%%%%%%%%%%

\begin{proof}
{\it Injectivity:} Let $y\in Y$ be fixed. Suppose that there are two different points $u$ and $e$ in $X$ such that $f(u)=f(e)=y$.   Since  $f$ is open with linear rate around $u$, there exist $\alpha >0$ and  $\epsilon>0$  such that:
\begin{equation}\label{linearopen} B(y;\alpha r)\subset f(B(u;r)), \mbox{ for all } 0<r<\epsilon.\end{equation}
Let $r\in (0,\epsilon)$ be small enough such that $f|_{B_r(u)}:B_r(u)\rightarrow f(B_r(u))$ is a homeomorphism,   and set $\rho = \alpha r>0$. Suppose first that $u=0$. We have that:
\begin{enumerate}
\item[\buletita] $F_y(0)=0\leq\rho$ and $F_y(e)=0\leq\rho$.
\item[\buletita] $|e|\geq r$, since $f|_{B_r(0)}$ is injective.
\item[\buletita] $F_y(x)\geq\rho$ for $|x|=r$, in view of \eqref{linearopen}.
\end{enumerate}
By Schechter-Katriel Mountain-Pass Theorem (see Theorem 7.2 in \cite{Ka}), there is a sequence $\{x_n\}\subset X$ such that $\lim_{n\rightarrow\infty }F_y(x_n)= c$ for some $c\geq\rho$ and satisfying \eqref{limiteacero}. Since $F_y$ satisfies the weighted Chang-Palais-Smale-condition, the sequence $\{x_n\}$ has a convergent subsequence $\{x_{n_k}\}$ with limit $\hat x$. Therefore $\lambda_{F_y}(\hat x)=0$,  and  $f(\hat x)\neq y$ since $\lim_{k\rightarrow\infty}F_y(x_{n_k})=F_y(\hat x)=c\geq\rho>0$. Therefore, we get a contradiction since:

\begin{claim}\label{claim}
For every $x\in X$, $\lambda_{F_y}(x)=0$ implies $f(x)=y$.
\end{claim}

In the case that $u\neq 0$,  we can consider $G_y(x)=F_y(u-x)$ instead of $F_y(x)$ and carry on an analogous reasoning.

\noindent {\it Surjectivity:} Let $y\in Y$ be fixed.  As we pointed out before, there is a minimizing  sequence $\{x_n\}\subset X$ such that $\lim_{n\rightarrow\infty}F_y(x_n)=\inf_X F_y$ and satisfying \eqref{limiteacero}. Since $F_y$ satisfies weighted Chang-Palais-Smale-condition, the sequence $\{x_n\}$ has a convergent subsequence $\{x_{n_k}\}$ with limit $\hat x$. As before, we have that $\hat x$ is a critical point of $F_y$. By Claim \ref{claim} we have that  $f(\hat x)=y$.
\end{proof}

\begin{proof}[Proof of Claim.]
Suppose that $\lambda_{F_y}(x)=0$ and  $f(x)\neq y$. Let $w^*\in\partial F_y(x)$. Since $Jf$ satisfies the strong chain rule condition $\partial F_y(x)\subset \Delta F_y(x)$. Then there is $y^*\in\partial \pmb{|}\cdot\pmb{|}(f(x)-y)$ and $T\in \overline{\rm co}\hspace{0.01in} Jf(x)$ such that  $w^*=y^*\circ T$. Since $f(x)-y\neq 0$ we have that $\pmb{|}y^*\pmb{|}_{Y^*}=1$. We have that
$$
|w^*|_{X^*}=|T^*y^*|_{X^*}\geq\inf_{\pmb{|}v^*\pmb{|}_{Y^*}=1}|T^*v^*|_{X^*}=C(T).
$$
Now, for every $\epsilon>0$ there is $T_\epsilon\in {\rm co}\hspace{0.01in} Jf(x)$ such that $\|T-T_\epsilon\|<\varepsilon$. Therefore, $C(T)=C^*(T)\geq C^*(T_\epsilon)-\epsilon\geq \reg Jf(x)-\epsilon$. So, we have that
$|w^*|_{X^*}\geq\reg Jf(x)$. Taking the minimum over $\partial F_y(x)$ we obtain that
\begin{equation} \label{claim-inequality}
\lambda_{F_y}(x)\geq \reg Jf(x).
\end{equation}
Therefore $\reg Jf(x)=0$ and we get contradiction.
\end{proof}

\begin{remark}\label{aclaracion}
In \cite{GUTU2018} and \cite{ISW}  the functional $G_y(x)=\frac{1}{2} F_y(x) ^2$ is considered instead of $F_y$. Suppose that $G_y$  satisfies the weighted Chang-Palais-Smale condition for some  weight $h$. Let $\{x_n\}$ any sequence in $X$ such that $\{F_y(x_n)\}$ is bounded and $\lambda_{F_y}(x_n)(1+h(|x_n|))=0$. Since $\lambda_{G_y}=F_y(x)\cdot \lambda_{F_y(x)}$ for all $x\in X$ and $y\in Y$, then $G_y(x_n)$ is bounded and $\lambda_{G_y}(x_n)(1+h(|x_n|))=0$. Therefore $\{x_n\}$ contains a (strongly) convergent subsequence.  So, if $G_y$ satisfies the weighted Chang-Palais-Smale condition for some  weight $h$ then $F_y$ satisfies the Chang-Palais-Smale condition with the same weight $h$.
\end{remark}

By Theorem \ref{main}, equations \eqref{lipschitzinequality} and \eqref{desigualdad2} we have:

%%%THEOREMA %%%
\begin{theorem}[{\bf Global Inverse Theorem II}]\label{global-2}
Let $(X,|\cdot|)$ and $(Y,\pmb{|}\cdot\pmb{|})$ be Banach spaces and let  $f:X\rightarrow Y$ be a locally Lipschitz map  with a pseudo-Jacobian $Jf$ regular at every $x\in X$ and satisfying the strong chain rule condition on $X$.  Suppose  that for every $y\in Y$  the locally Lipschitz functional  $F_y(x)=\pmb{|}f(x)-y\pmb{|}$ satisfies \ref{condicionchang}. Then $f$ is a norm-coercive homeomorphism locally bi-Lipschitz onto $Y$ with: $${\rm Lip} \, f^{-1}(y)\leq (\reg Jf(f^{-1}(y)))^{-1}.$$
 \end{theorem}
%%%%%%%%%%%%%%

\begin{remark}\label{combi}
Let $f:X\rightarrow Y$ be a locally Lipschitz map between Banach spaces  with a pseudo-Jacobian $Jf$ regular at every $x\in X$ and satisfying the strong chain rule condition on $X$. {\it If $f$ satisfies condition \ref{Ioffeconditionprima} for a positive nonincreasing and continuous function $m$  such that $\int_0^\infty m(\rho)d\rho=\infty$ then, for every $y\in Y$, the functional $F_y$ satisfies \ref{condicionchang} for the weight $$h(\rho):=\frac{m(0)}{m(\rho)}-1.$$}
\noindent Indeed, it is easy to verify that $h$ is actually a weight, namely, it is positive, nondecreasing, continuous map such that $\int_0^\infty\frac{1}{1+h(\rho)}d\rho=\infty$. Furthermore, for all $x\in X$:
$$0<m(0)<\reg Jf(x)(1+h(|x|)).$$
By the proof of (\ref{claim-inequality}) in the Claim above  we have that, if $y\in Y$ and $f(x)\neq y$, then $\lambda_{F_y}(x)\geq \reg Jf(x)$. Therefore,  if $f(x)\neq y$ then:
\begin{equation}\label{desigualdadclave}\lambda_{F_y}(x)(1+h(|x|))\geq m(0)>0.\end{equation}
Suppose that there is a sequence $\{x_n\}$ in $X$ such that $\lim_{n\rightarrow\infty} F_y(x_n)=c>0$ for some $y\in Y$. Then,  without loss of generality, we can assume that $f(x_n)\neq y$ for all natural $n$. Therefore, by \eqref{desigualdadclave} $\lim_{n\rightarrow \infty}\lambda_{F_y}(x_n)(1+h(x_n))$ can't be zero.  In other words, for each $y\in Y$, there is no sequence $\{x_n\}$ in $X$ such that $$\lim_{n\rightarrow\infty}F_y(x_n)=c>0 \mbox{ and }\lim_{n\rightarrow \infty}\lambda_{F_y}(x_n)(1+h(x_n))=0.$$
Note that, as we conclude in the first part of the proof of Theorem \ref{main},  this implies that $f$ is injective. Now, let $\{x_n\}\subset X$  be such that $\lim_{n\rightarrow\infty} F_y(x_n)=0$ and $\lim_{n\rightarrow\infty}\lambda_{F_y}(x_n)(1+h(|x_n|))=0$. Then, by \eqref{desigualdadclave} there exists  $m>0$ such that $f(x_n)=y$ for all $n\geq m$. Since $f$ is injective, this means that $x_n=f^{-1}f(x_n)=f^{-1}(y)$ for all $n\geq m$, so $\{x_n\}$ converges to $f^{-1}(y)$. Therefore \ref{condicionchang} is fulfilled.
\end{remark}

\section{Example}

In this section we give an example where the conditions for global invertibility introduced in the previous sections can be easily checked. We will be concerned with the following integro-differential equation, which has been considered,  with several variants, in \cite{ISW}, \cite{GK-15} and \cite{GK-16}:
\begin{equation}\label{integrodif}
x'(t) + \int_0^t \Phi(t, \tau, x(\tau)) \, d\tau = y(t), \quad \text{for a. e.} \,\, t\in [0,1]
\end{equation}
with initial condition:
\begin{equation}\label{integrodif-inicial}
x(0)=0.
\end{equation}
Here $y$ is a given function in the space $L^p[0, 1]$, where $1<p<\infty$ is fixed. It is natural to consider in this setting the space $W^{1,p}_0[0,1]$ of all absolutely continuous functions $x:[0,1] \rightarrow \mathbb R$ with $x(0)=0$ and such that $x'\in L^p[0,1]$. The space $W^{1,p}_0[0,1]$ is complete for the norm:
$$
\Vert x \Vert_{W^{1,p}_0} : = \left( \int_0^1 \vert x'(\tau) \vert^p \, d\tau \right)^{1/p}
$$

Then by a {\it solution} of the equation (\ref{integrodif}) with initial condition  (\ref{integrodif-inicial}) we mean a function $x\in W^{1,p}_0[0,1]$ satisfying (\ref{integrodif}) almost everywhere in $[0,1]$.

\

We denote $\Delta:=\{(t, \tau) \in [0,1]\times [0,1] \, : \, \tau \leq t\}$, and we will assume that the function $\Phi:\Delta \times \mathbb R \rightarrow \mathbb R$  satisfies the following conditions:
\begin{itemize}
\item [(i)] $\Phi(\cdot, \cdot, u)$ is measurable in $ \Delta $ for all $ u \in \mathbb R.$
\item [(ii)] There exist functions $a, b\in L^p(\Delta)$ such that
$$
\vert \Phi(t, \tau, u) \vert \leq a(t, \tau) \cdot \vert u \vert + b(t, \tau) \quad \text{for a. e.} \,\, (t, \tau) \in \Delta \, \, \,  \text{and all} \, \, u \in \mathbb R.
$$
\item [(iii)] There exists a continuous function $\theta:[0, \infty) \rightarrow (0,1)$ with the property that $\int_0^\infty (1-\theta(r)) dr =\infty$, and such that
$$
\vert \Phi(t, \tau, u) - \Phi(t, \tau, v)\vert \leq \theta(r) \vert u-v \vert \quad \text{for a. e.} \,\, (t, \tau) \in \Delta \, \, \,  \text{and all} \, \, \vert u \vert, \vert v \vert \leq r.
$$
\end{itemize}

Note that Condition (iii) is fulfilled, in particular, if there exists a constant $0<\theta <1$ such that $\Phi$ is globally $\theta$-Lipschitz in the third variable.

\begin{example}
Let $1<p<\infty$, and suppose that the function $\Phi$ satisfies conditions (i), (ii) and (iii) above. Then for each $y\in L^p[0, 1]$ there exists a unique solution of equation (\ref{integrodif}) with initial condition  (\ref{integrodif-inicial}) in the space $W^{1,p}_0[0,1]$.
\end{example}
\begin{proof}
Consider the Banach spaces $X=W^{1,p}_0[0,1]$ and $Y=L^p[0,1]$, and the map
$$
f:W^{1,p}_0[0,1] \rightarrow L^p[0, 1]
$$
defined as $f= T + g$, where $T, g:W^{1,p}_0[0,1] \rightarrow L^p[0, 1]$ are given respectively by
$$
T(x) = x'
$$
and
$$
g(x)(t) = \int_0^t \Phi(t, \tau, x(\tau)) \, d\tau.
$$
It is clear that $T$ is a  linear isomorphism which is, in fact, an isometry; that is,
$$
\Vert T(x) \Vert_{L^p} = \Vert x' \Vert_{L^p} = \left( \int_0^1 \vert x'(\tau) \vert^p \, d\tau \right)^{1/p} = \Vert x \Vert_{W^{1,p}_0}.
$$
Thus we have that $C(T) = \Vert T^{-1} \Vert^{-1} =1$.

On the other hand, we are next going to check that $g$ is Lipschitz on bounded subsets of $W^{1,p}_0[0,1]$. First note that, given $x\in W^{1,p}_0[0,1]$, for every $t\in [0,1]$ we have:
$$
\vert x(t) \vert = \left\vert \int_0^t x'(\tau) \, d\tau \right\vert \leq \int_0^1 \vert x'(\tau) \vert \, d\tau = \Vert x' \Vert_{L^1}\leq \Vert x' \Vert_{L^p} = \Vert x \Vert_{W^{1,p}_0},
$$
and therefore $\Vert x \Vert_{\infty} \leq \Vert x \Vert_{W^{1,p}_0}$.

Now let $r\geq 0$ and consider $u, v \in W^{1,p}_0[0,1]$ with $\Vert u \Vert_{W^{1,p}_0} \leq r$ and $\Vert v \Vert_{W^{1,p}_0} \leq r$. For each $t\in[0,1]$:
$$
\vert g(u)(t) - g(v)(t) \vert \leq \int_0^t \vert \Phi(t, \tau, u(\tau)) - \Phi(t, \tau, v(\tau))\vert \, d\tau
\leq \int_0^1 \theta (r) \cdot \vert u(\tau) - v(\tau) \vert \, d\tau
$$
$$
\leq \theta (r) \cdot \Vert u-v \Vert_{\infty} \leq \theta (r) \cdot  \Vert u-v \Vert_{W^{1,p}_0}.
$$
Then
$$
\Vert g(u) - g(v) \Vert_{L^p} = \left(\int_0^1 \vert g(u)(t) - g(v)(t) \vert^p d\,t \right)^{1/p} \leq  \theta (r) \cdot  \Vert u-v \Vert_{W^{1,p}_0}.
$$
This implies in particular that, for every $x\in W^{1,p}_0[0,1]$ with $\Vert x \Vert_{W^{1,p}_0}\leq r$ we have that ${\rm Lip} \, g(x)\leq \theta (r+\epsilon)$ for every $\epsilon>0$. By the continuity of $\theta$ we deduce that ${\rm Lip} \, g(x)\leq \theta (r)$ whenever $\Vert x \Vert_{W^{1,p}_0}\leq r$.

From Example \ref{sum} we have that $Jf(x):= T + {\rm Lip} \, g(x) \cdot \overline{B}_{L(X, Y)} $ is a pseudo-Jacobian of $f$, satisfying the strong chain rule condition. Let us see that $Jf$ is also  regular at every $x\in W^{1,p}_0[0,1]$. Indeed, suppose that $\Vert x \Vert_{W^{1,p}_0}\leq r$. For each $R\in {\mathcal L}(X, Y)$ with $\Vert R \Vert \leq {\rm Lip} \, g (x)$  we have $\Vert R \circ T^{-1}\Vert \leq {\rm Lip} \, g (x) \leq \theta (r)<1$. In particular, the operator ${\rm Id}_Y + R \circ T^{-1}$ is an isomorphism on $Y$. In this way we obtain that $T + R$ is an isomorphism. On the other hand, also using Example \ref{sum},  we have
$$\begin{aligned}
\reg  Jf (x) &= \sur Jf (x) = \inf \{C(T + R) \, : \, \Vert R \Vert \leq {\rm Lip} \, g(x)\}
\\
&\geq \inf \{C(T + R) \, : \, \Vert R \Vert \leq \theta (r)\} \geq 1-\theta (r) = m(r)>0,
\end{aligned}$$
where the continuous function $m(r):= 1-\theta (r)$ satisfies that $\int_0^\infty m(r) dr =\infty$.

Therefore, condition \ref{Ioffeconditionprima} and all the requirements of Theorem \ref{global-1} are satisfied,  and the desired conclusion follows. Also, from Remark \ref{combi} we see that, for each $y\in L^p[0,1]$, the functional $F_y$ defined in \eqref{funcionFY} satisfies the weighted Chang-Palais-Smale condition \ref{condicionchang} for the  weight
$$
h(r)=\frac{\theta (r)-\theta (0)}{1-\theta (r)}.
$$
Thus Theorem \ref{global-2} also applies in this case.
\end{proof}

\bibliographystyle{plainurl}
%\bibliography{referenciasDOI}

\end{document}